\documentclass[a4paper,11pt]{article}
\usepackage{fullpage}
\usepackage[utf8]{inputenc}
\usepackage[english]{babel} 

\usepackage{amsfonts, amssymb, amsbsy, amsmath, amsthm}
\usepackage{dsfont}
\usepackage{authblk}
\allowdisplaybreaks

\usepackage{enumitem}
\usepackage{multicol}
\usepackage{comment}
\usepackage[T1]{fontenc}

\usepackage[hyphens]{url}
\usepackage[hyphenbreaks]{breakurl}

\usepackage{mathtools}
\usepackage{xcolor}
\definecolor{darkblue}{rgb}{0,0,0.7}
\usepackage[colorlinks=true, allcolors=darkblue]{hyperref}
\usepackage[capitalise]{cleveref}
\crefname{algocf}{Algorithm}{Algorithms}
\crefname{equation}{Equation}{Equations} 
\crefname{figure}{Figure}{Figures}
\crefname{enumi}{}{}

\usepackage{tikz}
\usepackage{pgfplots}
\usetikzlibrary{intersections}

\hypersetup{
    colorlinks=true,
    breaklinks=true,            
}

\usepackage{pgfplots}
\usepackage{caption}
\usepackage{wrapfig}
\usepackage{graphicx}
\usepackage{tikz-cd}
\usepackage{tikz}
\usepackage{mathrsfs}
\usepackage{float}
\usetikzlibrary{shapes.geometric} 
\usetikzlibrary{positioning}

\newtheorem{theorem}{Theorem}[section]
\newtheorem{lemma}[theorem]{Lemma}

\newtheorem{proposition}[theorem]{Proposition}

\title{An Unsure Note on an Un-Schur Problem}

\author[1,2]{Olaf Parczyk}
\author[1,3]{Christoph Spiegel}

\affil[1]{\small Zuse Institute Berlin, Department AIS2T, \emph{lastname}@zib.de}
\affil[2]{\small Freie Universit\"at Berlin, Institute of Mathematics}
\affil[3]{\small Technische Universit\"at Berlin, Institute of Mathematics}

\date{}

\begin{document}

\maketitle

\begin{abstract}
    Graham, Rödl, and Ruciński~\cite{graham1996schur} originally posed the problem of determining the minimum number of monochromatic Schur triples that must appear in any 2-coloring of the first $n$ integers. This question was subsequently resolved independently by Datskovsky~\cite{datskovsky2003number}, Schoen~\cite{schoen1999number}, and Robertson and Zeilberger~\cite{robertson19982}. Here we suggest studying a natural anti-Ramsey variant of this question and establish the first non-trivial bounds by proving that the maximum fraction of Schur triples that can be rainbow in a given $3$-coloring of the first $n$ integers is at least $0.4$ and at most $0.66364$. We conjecture the lower bound to be tight. This question is also motivated by a famous analogous problem in graph theory due to Erd\H{o}s and S{\'o}s~\cite{erdos1972ramsey, nesetril2012mathematics} regarding the maximum number of rainbow triangles in any $3$-coloring of $K_n$, which was settled by Balogh et al.~\cite{Balogh2017}.
\end{abstract}

\section{Introduction}

In 1916, Schur proved that every $c$-coloring of $[n] := \{1, . . . , n\}$ contains a monochromatic Schur triple, that is a solution to the equation $x + y = z$, assuming $n$ is large enough depending on $c$. Having established the Ramsey-type property (before Ramsey's theorem had even been formulated), a natural question is to ask for the minimum number of Schur triples that any finite coloring of $[n]$ asymptotically has to contain. For a proof that supersaturation holds, that is that asymptotically some non-zero fraction of Schur triples needs to be monochromatic, see Frankl, Graham and R\"{o}dl~\cite{frankl1988quantitative}. Graham, R\"{o}dl, and Ruci\'nski~\cite{graham1996schur}, while studying the threshold behavior of random sets which commonly requires supersaturation-type arguments, observed the first non-trivial lower bound through a famous result of Goodman~\cite{goodman1959sets}. Chen and Graham proposed it as a \$100 problem at SOCA 96~\cite{robertson19982} and it was also conjectured that a construction attributed to Zeilberger should give the true minimum value. This conjecture was resolved in the affirmative independently by Datskovsky~\cite{datskovsky2003number}, Schoen~\cite{schoen1999number}, as well as Robertson and Zeilberger~\cite{robertson19982}. Datskovsky also showed the somewhat surprising results that the number of monochromatic Schur-triples in $\mathbb{Z}_n$ only depends on the cardinality of the color classes, a sort of strong Sidorenko-type property. Cameron, Cilleruelo, and Serra~\cite{Cameron2007} generalized this observation using purely combinatorial arguments to a much broader setting.

Seven decades after Schur's result, Alekseev and Savchev~\cite{alekseev1987problem} considered the {\lq}un-Schur{\rq} problem (a terminology attributed to Sands), proving that every equinumerous $3$-coloring of $[3n]$ must contain a rainbow Schur triple. The equinumerous requirement was later weakened, going as low as $1/4$ for the density of the smallest color class~\cite{schonheim1990partitions}. The \emph{maximum} number of rainbow Schur triples appears to have surprisingly not been studied in additive combinatorics, while the analogue question in graph theory has received a fair amount of attention~\cite{erdos1972ramsey, nesetril2012mathematics, Balogh2017, chao2023kruskal}.
Here we prove the following.
\begin{theorem}\label{thm:rainbow-schur-triple}
	The maximum fraction of Schur triples that can be rainbow in a $3$-coloring of the first $n$ integers asymptotically is between $0.4$ and $0.66364$.
\end{theorem}
The first main contribution of this work consists of an explicit construction giving the lower bound of $0.4$. It was found through exhaustive computations for small $n$ and interestingly combines both an interval-based and a  modular coloring approach. We conjecture the lower bound to be tight and the underlying construction unique. The second main contribution consists of establishing a non-trivial, though probably still far from tight, upper bound that can be thought of as a slightly strengthened but otherwise analogous approach to that taken in~\cite{graham1996schur}; in particular, our argument relies on the rainbow triangle result of Balogh et al.~\cite{Balogh2017} similar to how theirs relied on Goodman's results~\cite{goodman1959sets}. We discuss why other approaches do not seem to be easily applicable (motivating our title) in the last section.

\medskip
\noindent \textbf{Acknowledgements.} The research on this project was initiated during the 2023 research workshop of the Combinatorics and Graph Theory group at Freie Universität Berlin. We would like to thank Tibor Szabó for creating a nourishing research environment, both mentally and physically, as well as Micha Christoph and Cl{\'e}ment Requil{\'e} for initial discussions about the problem.

\section{Proof of \cref{thm:rainbow-schur-triple}} \label{sec:schurproof}

For us a Schur triple is an ordered triple, that is for $x + y = z$ we consider $(x, y, z)$ and $(y, x, z)$ as distinct.
This aligns with the notion in~\cite{datskovsky2003number} but differs from that in~\cite{robertson19982}, where $(x,y,z)$ and $(y,x,z)$ are the same Schur triple even if $x \neq y$, and that in~\cite{schoen1999number}, where additionally $x \neq y$. Asymptotically this of course has no impact up to a constant factor of $2$, which disappears when dealing with ratios.
One consequence is that we have exactly
\begin{equation}\label{eq:nr_of_schur_triples}
    \binom{n}{2} = \big( 1/2 + o(1) \big) \, n^2
\end{equation}
Schur triples in $[n]$, since for each $x, z \in [n]$ satisfying $x < z$ there exists a unique $y = z - x \in [n]$ completing the Schur triple $(x,y,z)$.

\subsection{The lower bound}
Combining an interval- and a mod-based approach, the coloring
\begin{equation*}
    c_0: [n] \to \{1, 2, 3\}, \quad i \mapsto \left\{\begin{array}{ll}
        1 & \text{if $i$ is odd and $i \le 2n/5$}\\
        2 & \text{if $i$ is odd and $i > 2n/5$}\\
        3 & \text{if $i$ is even}
        \end{array}\right.,
\end{equation*}
asymptotically gives $\big( 1/5 + o(1) \big) \, n^2$ rainbow Schur triples in $[n]$, i.e., a lower bound of $0.4$.
Indeed, fix an odd $x \le 2n/5$ (so $c_0(x)=1$). 
If $y$ is even ($c_0(y)=3$), then $x+y$ is odd, and the condition 
$2n/5 < x+y \le n$ ensures $c_0(x+y)=2$; 
the number of such even $y$ is $\tfrac{n-x}{2}-\tfrac{2n/5-x}{2}=\tfrac{3n}{10}$.
If $y$ is odd with $y>2n/5$ ($c_0(y)=2$) and $x+y\le n$, then $x+y$ is even and $c_0(x+y)=3$; 
the number of such $y$ is $\tfrac{n-x}{2}-\tfrac{2n/5}{2}=\tfrac{3n}{10}-\tfrac{x}{2}$.
Summing over all odd $x\le 2n/5$ and accounting for the same number of $(y,x,z)$ therefore gives
\[
2 \sum_{\substack{x\le 2n/5 \\ x\ \mathrm{odd}}}
\left(\frac{3n}{10}+\frac{3n}{10}-\frac{x}{2}\right)
=
\sum_{\substack{x\le 2n/5 \\ x\ \mathrm{odd}}}
\left(\frac{6n}{5}-x\right)
=
\left(\frac{1}{5}+o(1)\right)n^2.
\]

\subsection{The upper bound} For an arbitrary but fixed $n$ tending to infinity, let $S := \{(x, y, z) \in [n]^3 \mid x + y = z\}$ denote the set of all Schur triples in $[n]$ and $S^{(z)} := \{(x, y, z) \in S \}$ the set of all Schur triples whose largest element is some given $1 \le z \le n$. Clearly
\begin{equation}\label{eq:Szcard}
    |S^{(z)}| = z - 1.
\end{equation}
We let $c: [n] \to \{1,2,3\}$ be an arbitrary but fixed $3$-coloring of the first $n$ integers. We write $S_R^{(z)} = S_R^{(z)} (c) \subseteq S^{(z)}$ for the set of rainbow Schur triples in $S^{(z)}$ and $S_R = S_R(c) = \bigcup_{z=1}^n S_R^{(z)}$ for the set of all rainbow Schur triples when coloring $[n]$ with $c$. We also write $\operatorname{r}(z) = |S_R^{(z)}|$ and note that, by \cref{eq:Szcard}, obviously
\begin{equation}\label{eq:rzbnd}
    \operatorname{r}(z) \le |S^{(z)}| = z-1.
\end{equation}

\subsubsection*{Relating Schur triples to triangles}

Let $T = \left\{ (v_1, v_2, v_3) \mid v_1 < v_2 < v_3 \in [n + 1] \right\}$ denote the set of all triangles, identified by their \emph{ordered} vertices, in a complete graph with vertex set $[n + 1]$. Clearly
\begin{equation}\label{eq:Tcard}
    |T| = \binom{n+1}{3} = \big( 1/6 + o(1) \big) \, n^3.
\end{equation}
Let us define the function $f: T \to S, \, (v_1, v_2, v_3) \mapsto (v_2 - v_1, v_3 - v_2, v_3 - v_1)$, which is surjective and satisfies
\begin{equation}\label{eq:fmoscard}
    \left| f^{-1}(s) \right| = n + 1 - z \quad \text{for all } s \in S^{(z)}.
\end{equation}

We let the coloring $c$ of $[n]$ induce a coloring $c': E(K_{n+1}) \to \{1,2,3\}$ of the edges of $K_{n+1}$ through $c'(e) = c(\max e - \min e)$ for any $e = \{v_1, v_2\} \in E(K_{n+1})$ and write $T_R = T_R(c)$ for the set of rainbow triangles w.r.t. to that induced coloring. By Balogh et al.~\cite{Balogh2017}, we have
\begin{equation}\label{eq:balogh}
    |T_R| \le \big(2/5 + o(1) \big) \, |T| \overset{\eqref{eq:Tcard}}{=} \big(1/15 + o(1) \big) \, n^3.
\end{equation}
Note that two triangles in $T$ that are mapped to the same Schur triple by $f$ clearly are colored the same in $c'$. In particular, a triangle is rainbow with respect to $c'$ if and only if its image Schur triple is rainbow with respect to $c$, so that
\begin{equation}\label{eq:Trcard}
      \sum_{z = 1}^{n} \operatorname{r}(z) \, (n+1-z) \overset{\eqref{eq:fmoscard}}{=}  \sum_{s \in S_R} |f^{-1}(s)| = |T_R|\overset{\eqref{eq:balogh}}{\le} \big(1/15 + o(1) \big) \, n^3.
\end{equation}

\subsubsection*{A reweighing lemma}

Knowing the upper bound on $\sum_{z = 1}^{n} \operatorname{r}(z) \, (n+1-z)$ from \cref{eq:Trcard}, we want to upper bound $|S_R| = \sum_{z = 1}^{n} \operatorname{r}(z)$. For this we will make use of the following simple reweighing lemma.

\begin{lemma}\label{lemma:reweighting}
    Let $S$ be a finite set and $f, g: S \to \mathbb{R}_{\ge 0}$ positive functions. We have
    \begin{equation*}
        \sum_{s \in S} f(s) \le \sum_{s \in S_0} f_0(s)   
    \end{equation*}
    for any $S_0 \subseteq S$ and $f_0: S_0 \to \mathbb{R}$ satisfying (i) $g |_{S \setminus S_0} > 0$, (ii) $\max g |_{S_0} \le \min g |_{S \setminus S_0}$, (iii) $f |_{S_0} \le f_0$, and (iv) $\sum_{s \in S_0} f_0(s) \, g(s) \ge \sum_{s \in S} f(s) \, g(s)$.
\end{lemma}
\begin{proof}
    With $m_g:=\min g|_{S \setminus S_0}$ we have
    \begin{align*}
         m_g \, \left( \sum_{s \in S_0} f_0(s) - \sum_{s \in S} f(s) \right) & \ge \sum_{s \in S_0} \left( f_0(s) - f(s) \right) \, g(s) - \sum_{s \in S \setminus S_0} f(s) \, g(s) \ge 0,
    \end{align*}
    where the first inequality holds by (ii) and (iii) and the second inequality holds by (iv). We may divide by $m_g$ due to (i), giving the desired inequality.
\end{proof}
Before giving the complete argument, let us first look at a more simple application of this lemma. This exactly follows the ideas of Graham, R\"odl, Ruci{\'n}ski~\cite[p. 390]{graham1996schur} for the deterministic case. We choose $S = [n]$, $f(z) = \operatorname{r}(z)$, $g(z) = n + 1 - z$, $f_0(z) = z-1$, and $S_0 = \{ z \mid z \ge z_0 \}$, where $z_0 = z_0(n) \in [n]$ is chosen maximal such that
\begin{equation} \label{eq:defz0}
     \sum_{z = 1}^{n} \operatorname{r}(z) \, (n+1-z) \le \sum_{z = z_0}^{n} (z-1) \, (n+1-z),
\end{equation}
which is well-defined by \cref{eq:rzbnd}.
The conditions of the lemma are met since (i) holds trivially, (ii) holds since $g$ is monotone decreasing, (iii) holds by \cref{eq:rzbnd}, and (iv) holds by \cref{eq:defz0}. Writing $\alpha = \alpha(n) = z_0(n)/n$, \cref{lemma:reweighting} therefore implies that
\begin{equation}\label{eq:Sralphaub}
    |S_R| = \sum_{z=1}^n \operatorname{r}(z) \le \sum_{z=z_0}^n (z-1) = \big( 1/2 - \alpha^2/2 + o(1) \big) \, n^2.
\end{equation}
It remains to lower bound $\alpha$ in order to upper bound $|S_R|$. We note that 
\begin{equation*}
     \sum_{z = 1}^{n} \operatorname{r}(z) \, (n+1-z) > \!\!\! \sum_{z = z_0+1}^{n} (z-1) \, (n+1-z)
     = \big( 1/6 - \alpha^2/2 + \alpha^3/3 + o(1) \big) \, n^3,
\end{equation*}
where the strict inequality follows by choice of $z_0$ through \cref{eq:defz0}. Using \cref{eq:Trcard}, this implies $\alpha^3/3 - \alpha^2/2 + 1/10 \le o(1)$ and, since $0 \le \alpha \le 1$, therefore
\begin{equation*}
     \alpha \ge \sin \big( \pi/6 -\operatorname{atan}(2 \, \sqrt{6})/3 \big) + 1/2 + o(1) > 0.56706 + o(1).
\end{equation*}
By \cref{{eq:Sralphaub}}, it follows that
\begin{equation}
    |S_R| \le \big( 0.33922 + o(1) \big) \, n^2,
\end{equation}
which is an upper bound on the fraction of $0.67844$, slightly worse than the upper bound reported in \cref{thm:rainbow-schur-triple}. Let us now improve the argument to push the bound an epsilon below $2/3$.

\subsubsection*{A more nuanced reweighing}

In the simplified argument we assumed that all the triples in $S^{(z)}$ are rainbow for any $z_0 \le z \le n$. This is intuitively not possible and we can take advantage of that.
Let us add the notation $\operatorname{nr}(z) = |S^{(z)} \setminus S_R^{(z)}|$. By \cref{eq:Szcard} obviously $\operatorname{nr}(z) = z-1 - \operatorname{r}(z)$. For a given coloring $c: [n] \to \{1, 2, 3\}$ and $k_0 = k_0(n) \in [n]$ to be specified later, choose $z_0 = z_0(n, c, k_0)$ maximal s.t.
\begin{equation} \label{eq:defz0_new}
     \sum_{z = 1}^{n} \operatorname{r}(z) \, (n+1-z) 
     \le \sum_{z = z_0}^{n} (z-1) \, (n+1-z) - k_0 \, \sum_{z \in Z} (n+1-z),
\end{equation}
where
\begin{equation}\label{eq:Z}
    Z = Z(n, c, k_0) = \left\{ z \ge z_0 \mid \operatorname{nr} (z) \ge k_0 \right\} \subseteq [n].
\end{equation}
Like \cref{eq:defz0}, this is again well-defined by \cref{eq:rzbnd} and now also by our choice of $Z$.
This allows us to use the improved upper bound
\begin{equation*}
    f_0(z) = \left\{\begin{array}{ll}
        z - 1 - k_0  & \text{if } z \in Z,\\
        z - 1  & \text{otherwise.}
    \end{array}\right.
\end{equation*}
Writing $\alpha = \alpha(n, c, k_0) = z_0/n$, $\beta = \beta(n, c, k_0) = |Z|/n$, and $\gamma = \gamma(n) = k_0/n$, it follows that
\begin{equation} \label{eq:SRupper}
    |S_R| \le \left(\sum_{z =z_0}^{n} z-1 \right) - \sum_{z \in Z} k_0 = \left( 1/2 - \alpha^2/2 - \beta \, \gamma + o(1) \right) \, n^2.
\end{equation}
It remains to lower bound $\beta$ and $\alpha$, or more specifically $\alpha^2/2 + \beta \, \gamma$, dependent on our choice of $\gamma$. We then can choose $\gamma$ optimal to maximize this lower bound.
 
\paragraph*{First bound.} 

Let us write $C_i = \{ z \ge z_0 \mid z \notin Z, c(z) = i \}$ and $C_i' = \{ z < z_0 \mid c(z) = i \}$. Note that the set $[z_0 - 1]$ is the disjoint union of $C_1'$, $C_2'$, and $C_3'$, and $[n] \setminus [z_0 -1]$ is the disjoint union of $C_1$, $C_2$, $C_3$, and $Z$, so that
\begin{equation*}
    n = z_0 - 1 + |C_1| + |C_2| + |C_3| + |Z|.
\end{equation*}
We need to upper bound the cardinality of the sets $C_i$. Write $z_i = \max C_i$, where we follow the convention that $\max \emptyset = - \infty$. We have
\begin{equation}\label{eq:nr_ub}
    \operatorname{nr}(z_i) \le k_0 - 1,
\end{equation}
since for $z_i = -\infty$ the statement is trivially true and for $z_i \neq - \infty$ we have $z_i \notin Z$ but also $z_i \ge z_0$, so the statement follows by definition of $Z$ in \cref{eq:Z}. As every element $y \in (C_i' \cup C_i) \cap \{1, \ldots, z_i-1\} = (C_i' \cup C_i) \setminus \{z_i\}$ is uniquely responsible for a non-rainbow Schur triple $(x,y,z_i)$, it follows that $|C_{i}'| + |C_{i}| - 1 \le \operatorname{nr}(z_i)$, that is $|C_{i}'| + |C_{i}| \le k_0$ by \cref{eq:nr_ub}. In combination, this gives us $n \le \max \left( z_0 + |Z| + 2k_0, |Z| +  3k_0 \right)$ 
and therefore at least one of
\begin{equation}\label{eq:second_bound}
    \alpha + \beta  \ge 1 - 2 \gamma \quad \text{or} \quad \beta \ge 1 - 3\gamma
\end{equation}
has to hold.

\paragraph*{Second bound.} Since $z_0$ was again chosen maximal in \cref{eq:defz0_new}, we have
\begin{align*}
     \sum_{z = 1}^{n} \operatorname{r}(z) \, (n+1-z) & > \sum_{z = z_0+1}^{n} (z-1) \, (n+1-z)
     - k_0 \sum_{z \in Z \setminus \{z_0\}} (n+1-z) \\
     & \ge \sum_{z = z_0+1}^{n} (z-1) \, (n+1-z)
     - k_0 \sum_{z = z_0}^{z_0 + |Z|} (n+1-z) \\
     & = \big( 1/6 - \alpha^2/2 + \alpha^3/3 + (\alpha + \beta/2 - 1) \, \gamma \beta + o(1) \big) \, n^3.
\end{align*}
Using \cref{eq:Trcard}, this implies
\begin{equation}\label{eq:third_bound}
    \alpha^2/2 - \alpha^3/3 - (\alpha + \beta/2 - 1) \, \gamma \beta - 1/10 \ge o(1).
\end{equation}

\begin{figure}
    \centering
    \includegraphics[width=\linewidth]{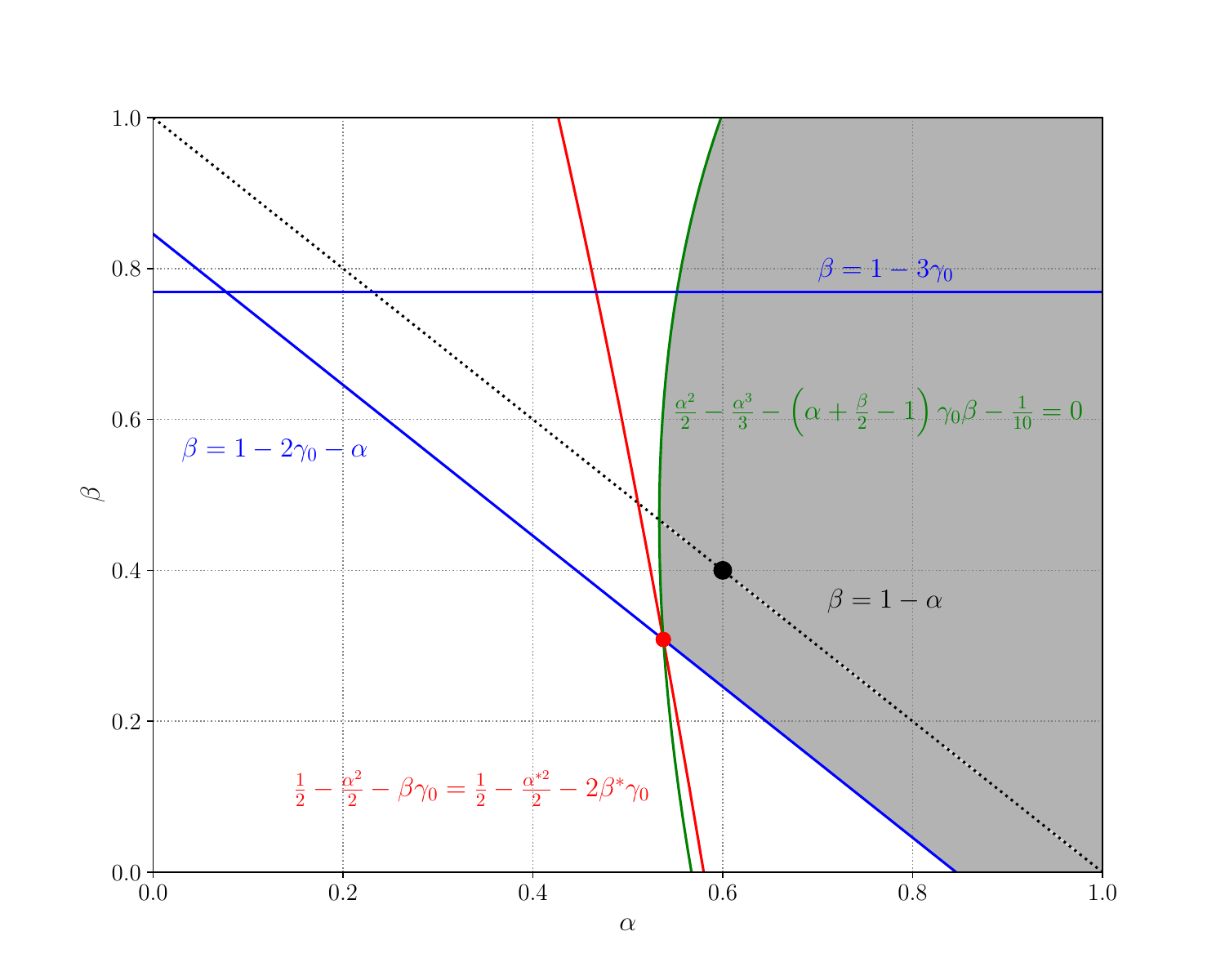}
    \caption{Region of feasible $\alpha$ and $\beta$ for fixed $\gamma_0= 0.077102$ with the first bound from \cref{eq:second_bound} in blue and the second bound from \cref{eq:third_bound} in green. The optimal point is in red with the corresponding contour line for the objective from \cref{eq:SRupper}. The black dot corresponds to our lower bound construction.}
    \label{fig:minmax}
\end{figure}

\paragraph*{Optimizing over the bounds.}
To bound the size of $S_R$ using \eqref{eq:SRupper}, it remains to solve
\begin{equation*}
    \min_{\gamma} \max \left\{ 1/2 - \alpha^2/2 -
    \beta \, \gamma \mid  0 \le \alpha, \beta \le 1 \text{ satisfying \eqref{eq:second_bound} and \eqref{eq:third_bound}} \right\},
\end{equation*}
where obviously $0 \le \gamma \le 1$. Choosing $\gamma_0 = 0.077102$, one can easily see that the first inequality in \cref{eq:second_bound} has to hold due to \cref{eq:third_bound}, see~\cref{fig:minmax}. The optimal values for $\alpha$ and $\beta$ for the fixed $\gamma_0$ therefore need to lie at the intersection of
\begin{equation}
    \beta_1 ( \alpha) = 1 - 2\gamma_0 - \alpha
\end{equation}
and 
\begin{equation}
    \beta_2 ( \alpha) = -(\alpha - 1) - \sqrt{(\alpha - 1)^2 + \frac{2}{\gamma_0} \left( \frac{\alpha^2}{2} - \frac{\alpha^3}{3} - \frac{1}{10} \right)}
\end{equation}
which is given by the unique root in $[0,1]$ of
\[
-\frac{2}{3\gamma_0}\alpha^3
+\left(1+\frac{1}{\gamma_0}\right)\alpha^2
-2\alpha
+\left(1-\frac{1}{5\gamma_0}-4\gamma_0^2\right)=0,
\]
which is 
\[
\alpha^\ast
=
\frac{538551}{1000000}
+
2\sqrt{\frac{638805538803}{3\cdot 10^{12}}}
\;
\cos\!\left(
\frac{1}{3}
\arccos\!\left(
\frac{366225273142527}{98258725691701849}
\right)
+\frac{4\pi}{3}
\right).
\]
With $\beta^\ast = 1 - 2\gamma_0 - \alpha^\ast$ an upper bound on the fraction of strictly less than $0.66364$ follows.

\section{Concluding remarks}

There are two clear avenues for improvement in the proof of the upper bound: first, the choice of $z_0$ and $Z$ is suboptimal because it neither accounts for the partially modular nature of the extremal construction nor captures the linear dependence of $\operatorname{nr} (z)$ on $z$ in that construction.
The first bound therefore falls short, as can be seen in \cref{fig:minmax}, and 
modifying this part of the argument by potentially improving it up to $\alpha \ge 1 - \beta$ for the right choice of $\gamma$ would probably yield the biggest immediate improvement.
Second, the application of the bound due Balogh et al.~\cite{Balogh2017} ignores that the family of $3$-edge-colorings of $K_{n+1}$ induced by $3$-colorings of $[n]$ does not include the known extremal coloring maximizing the number of rainbow triangles. In particular, the coloring $c_0$ establishing the lower bound gives $|T_R(c_0)| = \big( 0.05\overline{3} + o(1) \big) \, n^3$, whereas the general upper bound in \cref{eq:Trcard} had a constant of $0.0\overline{6}$. Obtaining a tight upper bound would therefore probably require completely avoiding the translation of the problem into an auxiliary graph theoretic formulation.

This begs the question whether any of the three approaches that have been successful for resolving the minimum number of monochromatic Schur triples could perhaps also be helpful here. Robertson and Zeilberger~\cite{robertson19982} reformulated the problem as the discrete optimization problem of determining the minimum of
\begin{equation*}
    F (x_1, \ldots , x_n) = \sum_{\substack{1  \le i < j \le n \\ i + j \le n}} \left( x_i x_j x_{i+j} + (1 - x_i)(1 - x_j)(1 - x_{i+j}) \right)
\end{equation*}
over the $n$-dimensional hypercube $\{0, 1\}^n$ and then (in part computationally) studied its discrete partial derivatives. This approach does not seem to readily extend to more than two colors without the formulation of the underlying optimization problem becoming unwieldy.
Schoen~\cite{schoen1999number} and Datskovsky~\cite{datskovsky2003number} on the other hand reformulate the problem using the Fourier transform $f_k(x) = \exp (2 \pi i k x) $, noting that the number of monochromatic Schur triples in a coloring of $[n]$ expressed as a bipartition $[n] = R \sqcup B$ can be counted through 
\begin{equation*}
    \int_0^1 f_R(x) \, f_R(x) \, f_R(-x) + f_B(x) \, f_B(x) \, f_B(-x) \, dx,
\end{equation*}
where $f_R(x) = \sum_{k \in R} f_k(x)$ and $f_B(x) = \sum_{k \in B} f_k(x)$. Briefly switching to $\mathbb{Z}_n$, this yields the the very simple expression of $\big( |R|^3 + |B|^3 \big) / n$ for the number of monochromatic Schur triples. The difficulty in extending  this approach in order to study rainbow Schur triples likewise seems to lie in the increase of the number of colors; Cameron, Cilleruelo, and Serra~\cite{Cameron2007} determined that in a three coloring of $\mathbb{Z}_n$ expressed as a tripartition $[n] = R \sqcup G \sqcup B$ the number of monochromatic Schur triples is  $\big( 3 |R|^2 + 3 |G|^2 + 3 |B|^2 - n^2 + \operatorname{r} \big) / 2$, where $\operatorname{r}$ denotes the number of rainbow Schur triples, but no approach to obtain a tight lower bound on the minimum number of monochromatic Schur triples based on this has been proposed. It should be noted that the corresponding problem in graph theory, determining the minimum number of monochromatic triangles, has been resolved for three and four colors~\cite{cummings2013monochromatic, kiem2023four}. Finally, directly applying ideas behind the techniques used by Balogh et al.~\cite{Balogh2017} in the additive setting brings its own set of problems, since formulating flag algebras for anything other than additive structures in $\mathbb{F}_q^n$ has proven challenging~\cite{rue2023rado}.

\medskip

It of course also makes sense to ask a similar question for other additive structures besides Schur triples, most notably for $k$-term arithmetic progressions ($k$-AP), that is for tuples $(x, x+d, \ldots, x+(k-1) \, d)$ with common difference $k \ge 1$.
Radoi\v{c}i\'c asked under what density requirements a three-coloring of $[n]$ must contain for $3$-term arithmetic progressions ($3$-AP), that is a solution to $x + y = 2z$. This was resolved by showing that the smallest color class can have density as low as $1/6$ while still guaranteeing the existence of rainbow $3$-AP~\cite{jungic2003rainbow, butler2014rainbow, axenovich2004rainbow}. Regarding the maximum number of rainbow arithmetic progressions, the only direct reference we could find was by Jungi\'c et al.~\cite{jungic2003rainbow}, who state in the concluding section of their paper that ``(i)t is easy to show that the maximal number of rainbow AP(3)s over all equinumerous 3-colorings of $[3n]$ is $\lfloor 3n^2/2 \rfloor $, this being achieved for the unique 3-coloring with color classes $R = \{n \mid n \equiv 0 \bmod 3 \}$, $B = \{ n \mid n \equiv 1 \bmod 3 \}$ and $G = \{ n | n \equiv 2 \bmod 3\}$.'' Here we state the obvious generalizations of both the upper and lower bound, which align when $k$ is prime but otherwise leave a (significant) gap.
We again conjecture that the lower bound is tight and the underlying construction (with the obvious exception of $k = 2$) unique.
\begin{proposition}\label{prop:rainbow-ap}%
    The maximum fraction of $k$-term arithmetic progressions that can be rainbow in a $k$-coloring of the first $n$ integers asymptotically is between $\prod_{i=1}^m(1 - 1/p_i)$ and $(1 - 1/k)$ if $k$ has the prime decomposition $k = p_1^{a_1} \cdots p_m^{a_m}$.
\end{proposition}
\begin{proof}
    Any choice of $1 \le x_1 < x_k \le n$ outlines a $k$-term arithmetic progression if and only if $k - 1$ divides $x_k - x_1$, in which case that arithmetic progression is $(x_1, x_1 + d, x_1 + 2d, \ldots, x_1 + (k-1) \, d)$ with  $x_1 + (k-1) \, d = x_k$ and $d = (x_k - x_1) / (k-1)$. As a consequence, we have $n - (k-1) \, d$ arithmetic progressions with common difference $1 \le d \le \lfloor (n-1)/(k-1) \rfloor$, and therefore
    \begin{equation}\label{eq:nr_of_aps}
        \big( 1/(2k-2) + o(1) \big) \, n^2
    \end{equation}
    arithmetic progressions in $[n]$ overall.
        
    \paragraph{Lower bound} Coloring $[n]$ with $c: [n] \to \{1, \ldots, k\}, \, i \mapsto i + 1 \bmod k$, any arithmetic progression with common difference $d$ is rainbow if and only if $\gcd(d, k) = 1$. Using Euler's totient function, it follows that a $\prod_{i = 1}^m (1 - 1/p_i)$ fraction of all arithmetic progressions is rainbow.

    \paragraph{Upper bound} Fix some coloring $c : [n] \rightarrow \{1, \ldots, k\}$. For $1 \le i \le k$ and $0 \le r \le k-2$, let $C_{i, r} = \{x \in [n] \mid c(x) = i \text{ and } r = x \bmod k - 1\}$. We want to count rainbow $k$-term arithmetic progressions $(x_1, x_2, \dots, x_k)$ when coloring $[n]$ with $c$.
    We have $n$ choices for $x_1$ and, briefly allowing for $x_k < x_1$, we have $\sum_{i \neq c(x_1)} |C_{i, r}|$ choices for $x_k$ if $r = x_1 \bmod k-1$. Correcting for $x_k > x_1$ through an overall factor of $1/2$, the number of rainbow $k$-term arithmetic progressions is therefore upper bounded by
    \[
       \left( \sum_{i, r} |C_{i, r}| \sum_{i' \not = i} |C_{i', r}| \right) / 2 = \sum_{i\not=i',r} |C_{i, r}|  |C_{i', r}| / 2 \le \frac{n^2}{2k},
    \]
    where the inequality holds since by Cauchy-Schwarz
    \[
    \sum_{i \not= i'} |C_{i, r}|  |C_{i', r}| = \left(\sum |C_{i, r}|\right)^2 - \sum |C_{i, r}|^2 \le \frac{k-1}{k}\left(\sum |C_{i, r}|\right)^2 = \frac{n^2}{k(k-1)}
    \]
    for any $r$. This gives an upper bound on the fraction of $(k-1)/k$.
\end{proof}

We hope this work motivates further research into these types of additive anti-Ramsey questions and believe that the upper bounds, both in \cref{thm:rainbow-schur-triple} and in \cref{prop:rainbow-ap}, are amenable to improvements using alternative techniques.


\begin{thebibliography}{99}

\bibitem{alekseev1987problem}
Vladimir E Alekseev and Stoyan Savchev.
\newblock Problem {M}. 1040.
\newblock {\em Kvant}, 4(23), 1987.

\bibitem{axenovich2004rainbow}
Maria Axenovich and Dmitri Fon-Der-Flaass.
\newblock On rainbow arithmetic progressions.
\newblock {\em Electronic Journal of Combinatorics}, 11:R1, 2004.

\bibitem{Balogh2017}
József Balogh, Ping Hu, Bernhard Lidický, Florian Pfender, Jan Volec, and Michael Young.
\newblock Rainbow triangles in three-colored graphs.
\newblock {\em Journal of Combinatorial Theory, Series B}, 126:83--113, 2017.

\bibitem{butler2014rainbow}
Steve Butler, Craig Erickson, Leslie Hogben, Kirsten Hogenson, Lucas Kramer,
Richard L Kramer, Jephian Chin-Hung Lin, Ryan R Martin, Derrick Stolee,
Nathan Warnberg, and Michael Young.
\newblock Rainbow arithmetic progressions.
\newblock {\em Journal of Combinatorics}, 7(4):595--626, 2016.

\bibitem{Cameron2007}
Peter~J Cameron, Javier Cilleruelo, and Oriol Serra.
\newblock On monochromatic solutions of equations in groups.
\newblock {\em Revista Matem{\'a}tica Iberoamericana}, 23(1):385--395, 2007.

\bibitem{chao2023kruskal}
Ting-Wei Chao and Hung-Hsun~Hans Yu.
\newblock Kruskal--Katona-type problems via the entropy method.
\newblock {\em Journal of Combinatorial Theory, Series B}, 169:480--506, 2024.

\bibitem{cummings2013monochromatic}
James Cummings, Daniel Kr{\'a}ľ, Florian Pfender, Konrad Sperfeld, Andrew Treglown, and Michael Young.
\newblock Monochromatic triangles in three-coloured graphs.
\newblock {\em Journal of Combinatorial Theory, Series B}, 103(4):489--503, 2013.

\bibitem{datskovsky2003number}
Boris~A Datskovsky.
\newblock On the number of monochromatic {S}chur triples.
\newblock {\em Advances in Applied Mathematics}, 31(1):193--198, 2003.

\bibitem{erdos1972ramsey}
Paul Erd\H{o}s and Andr{\'a}s Hajnal.
\newblock On {R}amsey like theorems. problems and results.
\newblock In {\em Combinatorics (Proc. Conf. Combinatorial Math., Math. Inst., Oxford, 1972)}, 123--140, 1972.

\bibitem{frankl1988quantitative}
Peter Frankl, Ronald~L Graham, and Vojtech R{\"o}dl.
\newblock Quantitative theorems for regular systems of equations.
\newblock {\em Journal of Combinatorial Theory, Series A}, 47(2):246--261, 1988.

\bibitem{goodman1959sets}
Adolph~W Goodman.
\newblock On sets of acquaintances and strangers at any party.
\newblock {\em The American Mathematical Monthly}, 66(9):778--783, 1959.

\bibitem{graham1996schur}
Ronald Graham, Vojtech R{\"o}dl, and Andrzej Ruci{\'n}ski.
\newblock On {S}chur properties of random subsets of integers.
\newblock {\em Journal of Number Theory}, 61(2):388--408, 1996.

\bibitem{jungic2003rainbow}
Veselin Jungic, Jacob Licht, Mohammad Mahdian, Jaroslav Nešetřil, and Rados Radoicic.
\newblock Rainbow arithmetic progressions and anti-{R}amsey results.
\newblock {\em Combinatorics, Probability and Computing}, 12(5-6):599--620, 2003.

\bibitem{kiem2023four}
Aldo Kiem, Sebastian Pokutta, and Christoph Spiegel.
\newblock The four-color {R}amsey multiplicity of triangles.
\newblock {\em arXiv preprint arXiv:2312.08049}, 2023.

\bibitem{nesetril2012mathematics}
Jaroslav Nešetřil and Vojtech R{\"o}dl.
\newblock {\em Mathematics of {R}amsey theory}, volume~5.
\newblock Springer Science \& Business Media, 2012.

\bibitem{robertson19982}
Aaron Robertson and Doron Zeilberger.
\newblock A $2$-coloring of $[1, n]$ can have $(1/22) n^2+ o (n)$ monochromatic {S}chur triples, but not less!
\newblock {\em Electronic Journal of Combinatorics}, 5:R19, 1998.

\bibitem{rue2023rado}
Juanjo Ru{\'e} and Christoph Spiegel.
\newblock The {R}ado multiplicity problem in vector spaces over finite fields.
\newblock {\em Finite Fields and Their Applications}, 111:102782, 2026.

\bibitem{schoen1999number}
Tomasz Schoen.
\newblock The number of monochromatic {S}chur triples.
\newblock {\em European Journal of Combinatorics}, 20(8):855--866, 1999.

\bibitem{schonheim1990partitions}
Joseph Schönheim.
\newblock On partitions of the positive integers with no $x, y, z$ belonging to distinct classes satisfying $x+y=z$.
\newblock In {\em Number theory (Banff, AB, 1988)}, 
de Gruyter, Berlin, 515--528, 1990.

\end{thebibliography}
\end{document}